\documentclass[12pt]{article}

\usepackage[T1]{fontenc}
\usepackage[utf8]{inputenc}
\usepackage{amsthm, amssymb}
\usepackage{mathtools}
\usepackage{mathabx}
\usepackage{dsfont}
\usepackage{accents}
\usepackage{graphicx}
\usepackage{authblk}
\usepackage{xcolor}

\newcommand{\ind}{\mathds{1}}
\newcommand{\inner}[2]{\langle #1, #2\rangle}
\renewcommand{\div}[1]{\text{div}\,(#1)}

\newcommand{\R}{\mathbb{R}}
\let\goodlabel=\label
\newcommand{\inlabel}[1]{\refstepcounter{equation}\goodlabel{#1}(\theequation)}

\DeclarePairedDelimiter\abs{\lvert}{\rvert}
\DeclarePairedDelimiter\norm{\lVert}{\rVert}
\DeclarePairedDelimiter\japan{\langle}{\rangle}

\def\XXint#1#2#3{{\setbox0=\hbox{$#1{#2#3}{\int}$ }
\vcenter{\hbox{$#2#3$ }}\kern-.6\wd0}}

\theoremstyle{plain}
  \newtheorem{theorem}{Theorem}
  \newtheorem{corollary}[theorem]{Corollary}
  \newtheorem{lemma}[theorem]{Lemma}
  
\theoremstyle{definition}

\newtheorem{innercustomthm}{Theorem}
\newenvironment{customthm}[1]
  {\renewcommand\theinnercustomthm{#1}\innercustomthm}
  {\endinnercustomthm}

\title{Recovery of the Derivative of the Conductivity at the Boundary}

\author[ ]{Ponce-Vanegas, Felipe}

\affil[ ]{\small BCAM - Basque Center for Applied Mathematics}
\affil[ ]{\tt fponce@bcamath.org}

\date{}

\begin{document}
\maketitle

\begin{abstract}
We describe a method to reconstruct the conductivity and its normal derivative at the boundary from the knowledge of the potential and current measured at the boundary. This boundary determination implies the uniqueness of the conductivity in the bulk when it lies in $W^{1+\frac{n-5}{2p}+,p}$, for dimensions $n\ge 5$ and for $n\le p<\infty$.  
\end{abstract}

Electrical Impedance Imaging is a technique to recover the conductivity in the bulk of a body from measurements of potential and current at the boundary. The potential $u$ in a domain $\Omega\subset\R^n$ satisfies the equation
\begin{equation}\label{eq:BVP}
\begin{aligned}
\div{\gamma\nabla u} &= 0 \\
u|_{\partial\Omega} &= f,
\end{aligned}
\end{equation}
where $\gamma$ is the conductivity and $f\in H^\frac{1}{2}(\partial\Omega)$ is the potential at the boundary --the definitions of the spaces used here are placed at the end of the article. The conductivity satisfies the condition $0<c\le \gamma\le C$. The current measured at the boundary is $\gamma\partial_\nu u|_{\partial\Omega}$, where $\nu$ is the outward-pointing normal vector. The operator $\Lambda_\gamma$ that maps $u|_{\partial\Omega}$ to $\gamma\partial_\nu u|_{\partial\Omega}$ is known as the Dirichlet-to-Neumann map, and it is defined as the functional $\Lambda_\gamma:H^\frac{1}{2}(\partial\Omega)\mapsto H^{-\frac{1}{2}}(\partial\Omega)$ given by
\begin{equation*}
\inner{\Lambda_\gamma f}{g} := \int_{\Omega}\gamma\nabla u\cdot\nabla v,
\end{equation*}
where $u$ solves the boundary value problem \eqref{eq:BVP} and $v\in H^1(\Omega)$ is \textit{any} extension of $g\in H^\frac{1}{2}(\partial\Omega)$. If we choose $v$ such that $\div{\gamma\nabla v}=0$, then we see that $\Lambda_\gamma$ is symmetric.

In \cite{MR590275} Calderón posed the problem of deciding whether the conductivity can be uniquely recovered from the data at the boundary, \textit{i.e.} whether $\Lambda_{\gamma_1}=\Lambda_{\gamma_2}$ implies that $\gamma_1=\gamma_2$. One of the earliest results, due to Kohn and Vogelius in \cite{MR739921}, is that $\Lambda_{\gamma_1}=\Lambda_{\gamma_2}$ implies that $\partial^N_\nu\gamma_1=\partial^N_\nu\gamma_2$ at the boundary for every $N$ when $\gamma_1,\gamma_2\in C^\infty$. Sylvester and Uhlmann \cite{MR873380} made use of this result to prove uniqueness in the bulk for $C^2$ conductivities.

It is hard to prove uniqueness in the bulk, so uniqueness at the boundary may be considered as a toy problem; moreover, many proofs of inner uniqueness use uniqueness at the boundary as the first step, and this is in fact the motivation behind this article. For $\gamma_1,\gamma_2\in W^{s,p}(\Omega)$, some arguments need to extend the conductivities $\gamma_1$ and $\gamma_2$ to the whole space in such a way that $\gamma_1=\gamma_2$ in $\R^n\backslash\Omega$, and $\gamma_1,\gamma_2\in W^{s,p}(\R^n)$. Brown proved in \cite{MR1881563} that, under mild conditions of regularity, the conductivity can be recovered at the boundary. In particular, if $\gamma_1,\gamma_2\in W^{s,p}(\Omega)$ for $1\le s\le 1+\frac{1}{p}$ and $p\ge n$, then $\Lambda_{\gamma_1}=\Lambda_{\gamma_2}$ implies that $\gamma_1=\gamma_2$ at $\partial\Omega$, and then, by function space arguments, $\gamma_1$ and $\gamma_2$ can be adequately extended to $\R^n$; for details the reader is referred to \cite{MR2026763}. The possibility of this extension was used by Haberman \cite{MR3397029}, and by Ham, Kwon and Lee \cite{HKL} to prove uniqueness inside $\Omega\subset\R^n$ when $3\le n\le 6$.

When $\gamma_1,\gamma_2\in W^{s,p}(\Omega)$ for $s>1+\frac{1}{p}$, then to extend both conductivities adequately to $\R^n$ the condition $\partial_\nu\gamma_1=\partial_\nu\gamma_2$ at $\partial\Omega$ is necessary. The result of Kohn and Vogelius holds for smooth conductivities, so we cannot use it with rough conductivities. The main result of this article is that, under mild conditions of regularity, the conductivity and its normal derivative at the boundary is uniquely determined by $\Lambda_\gamma$; furthermore, our theorem provides a method of reconstruction.

\begin{theorem}\label{thm:recovery}
Suppose that $0<c\le\gamma\le C$ and that $\Omega\subset\R^n$ is a Lipschitz domain. 

{\bf (A)} If $\gamma\in W^{s,p}(\Omega)$ for $s>\frac{1}{p}$ and $1<p<\infty$, then for $y\in\partial\Omega$ a.e. there exist a family of functions $f_{0,h}$ and constants $c_{0,h}\sim 1$ such that
\begin{equation}\label{eq:thm_recovery_A}
\inner{\Lambda_\gamma f_{0,h}}{f_{0,h}} = c_{0,h}\gamma(y)+o(1)\quad \text{as } h\to 0.
\end{equation}
The constants $c_{0,h}$ do not depend on the conductivity.

{\bf (B)} If $\gamma\in W^{s,p}(\Omega)$ for $s>1+\frac{1}{p}$ and $2\le p<\infty$, or for $s>\frac{3}{p}$ and $1<p<2$, then for $y\in\partial\Omega$ a.e. there exist a family of functions $f_{1,h}$ and constants $c_{0,h},c_{1,h}\sim 1$ such that
\begin{equation}\label{eq:thm_recovery_B}
\inner{\Lambda_\gamma f_{1,h}}{f_{1,h}}-c_{0,h} = c_{1,h}\partial_\nu\log\gamma(y)h+o(h)\quad \text{as } h\to 0.
\end{equation}
The constants $c_{0,h}$ and $c_{1,h}$ do not depend on the conductivity.
\end{theorem} 

No attempt is made to get the best error terms implicitly involved in \eqref{eq:thm_recovery_A} and \eqref{eq:thm_recovery_B}.

The condition $\gamma\in W^{s,p}(\Omega)$, for $s>l+\frac{1}{p}$ and $l=0$ or 1, is the lowest regularity needed to make sense of the trace values of $\gamma$ and of $\partial_\nu \gamma$ respectively. In fact, by the trace theorem $\norm{\partial^l_\nu\gamma}_{L^p(\partial\Omega)}\le C\norm{\gamma}_{s,p}$ if $s>l+\frac{1}{p}$; the reader is referred to \cite{MR884984, MR781540} for details. 

The proof is mainly inspired by the work of Brown \cite{MR1881563}, who used highly oscillatory solutions to recover the value of $\gamma|_{\partial\Omega}$. We borrow many of his arguments, but we do not use oscillatory solutions, instead we follow Alessandrini \cite{MR1047569} and use singular solutions.

The motivation of the proof comes from the expansion, at least in the smooth class, $\Lambda_\gamma \sim \lambda^{1}+\lambda^0+\cdots$, where $\lambda^i\in S^i$ are pseudo-differential operators. This was proved by Sylvester and Uhlmann in \cite{MR924684}, and they showed that the information about $\partial^l_\nu\gamma$ at $\partial\Omega$ can be extracted from $\lambda^{1-l}$. Therefore, we try to use approximated solutions of \eqref{eq:BVP}, so that the boundary data $f$ concentrates as a Dirac's delta at some point on the boundary, and heuristically we get $\Lambda_\gamma(\delta_0)$. We follow this argument in Section~\ref{sec:Recovery}.

The main tool in our investigation is an approximation property at almost every point on the boundary. We did not find a suitable reference to the approximation we needed, so we include a proof here in Section~\ref{sec:Approximation}.

\begin{theorem}\label{thm:Boundary_Approximation}
Suppose that $\Omega\subset\R^n$ is a domain with Lipschitz boundary. If $f\in B^{s,p}(\Omega)$ for $1+\frac{1}{p}>s>\frac{1}{p}$, then for $0\le\alpha<s-\frac{1}{p}$ and for $y\in\partial\Omega$ a.e. it holds that
\begin{equation}
\Big(\frac{1}{r^n}\int_{B_r(y)\cap\Omega}\abs{f(x)-f(y)}^q\,dx\Big)^\frac{1}{q} \le Cr^\alpha,\quad \text{where } r\le 1 \text{ and } 1\le q\le p.
\end{equation}
The constant $C$ depends on $y$.
\end{theorem} 

As a consequence of Theorem~\ref{thm:recovery} and a result of the author in \cite[Thm. 4]{PV}, we get the following theorem.
\begin{theorem}
For $n\ge 5$ suppose that $\Omega\subset\R^n$ is a Lipschitz domain. If $\gamma_1$ and $\gamma_2$ are in $W^{1+\frac{n-5}{2p}+, p}(\Omega)\cap L^\infty(\Omega)$ for $n\le p<\infty$, and if $\gamma_1,\gamma_2\ge c>0$, then 
\begin{equation}
\Lambda_{\gamma_1}=\Lambda_{\gamma_2}\text{ implies that } \gamma_1=\gamma_2. 
\end{equation}
\end{theorem}

The reader can consult the symbols and notations used at the end of the article.

\subsection*{Acknowledgments}

I thank Pedro Caro for sharing his many insights with me. This research is supported by the Basque Government through the BERC 2018-2021 program, and by the Spanish State Research Agency through BCAM Severo Ochoa excellence accreditation SEV-2017-0718 and through projects ERCEA Advanced Grant 2014 669689 - HADE and PGC2018-094528-B-I00.

\section{Reconstruction at the Boundary}\label{sec:Recovery}

We assume in this section that $0\in\partial\Omega$ and that Theorem~\ref{thm:Boundary_Approximation} holds for $0$ every time we use it or a variant of it. We assume also that there is a ball $B_\delta(0)$ and a Lipschitz function $\psi$ such that
\begin{equation*}
B_\delta\cap\Omega = \{(x',x_n)\in B_\delta \mid \psi(x')<x_n \}.
\end{equation*}
We assume that $\psi(0)=\nabla\psi(0)=0$ and that $-e_n$ is a Lebesgue point of the outward-pointing normal vector $\nu:=(1+\abs{\nabla\psi}^2)^{-\frac{1}{2}}(\nabla\psi, -1)$.  

\subsection{Value at the Boundary}

The reconstruction of $\gamma|_{\partial\Omega}$ is based on the function $u(x):=x_n/\abs{x}^n$, which solves the boundary value problem $\Delta u = 0$ in the upper half-plane $H_+$, with $u|_{\partial H_+}=c\delta_0$. Since $\gamma\in W^{s,p}\cap L^\infty$, then by Gagliardo-Nirenberg, see \textit{e.g.} \cite{MR3813967}, we can assume that $\gamma\in W^{s,p}(\Omega)$ for $s>\frac{1}{p}$ and $2\le p<\infty$. 

For $h\ll 1$ we define the approximated solutions $u_h(x):=u(x+he_d)$ of \eqref{eq:BVP}, and we define the correction functions $r_h\in H^1_0(\Omega)$ such that
\begin{equation*}
\div{\gamma\nabla(u_h+r_h)} = 0.
\end{equation*}
Thus, we have that
\begin{equation*}
h^n\inner{\Lambda_\gamma (u_h|_{\partial\Omega})}{u_h|_{\partial\Omega}} = h^n\int \gamma\nabla (u_h+r_h)\cdot\nabla u_h.
\end{equation*}
The term $h^n$ is a normalization factor, and the functions $f_{0,h}$ in Theorem~\ref{thm:recovery}(A) are $f_{0,h}:=h^\frac{n}{2}u_h|_{\partial\Omega}$. 

The main part of the integral above is $\int\gamma\nabla u_h\cdot\nabla u_h$, and we extract the value of $\gamma(0)$ from it. We use the dilation $(\nabla u_h)(hx)= h^{-n}\nabla u_1(x)$ to get
\begin{align*}
 \int_\Omega\gamma\nabla u_h\cdot\nabla u_h &= \gamma(0) \int_\Omega\abs{\nabla u_h}^2+ \int_\Omega(\gamma-\gamma(0))\abs{\nabla u_h}^2 \\
&=\gamma(0)h^{-n}\int_{h^{-1}\Omega}\abs{\nabla u_1}^2+ \int_\Omega(\gamma-\gamma(0))\abs{\nabla u_h}^2 \\
&= A_1 + A_2.
\end{align*}
We set the first term as $h^nA_1=c_{0,h}\gamma(0)$. To control $A_2$ we bound it as
\begin{equation*}
\abs{A_2}\le C\int_{\Omega}\abs{\gamma(x)-\gamma(0)}\frac{dx}{\abs{x+he_n}^{2n}}.
\end{equation*}  
When $\abs{x}\ge 5h$ we see that $\abs{x}\sim \abs{x+he_n}$. When $\abs{x}<5h$ we exploit the Lipschitz regularity of the boundary, and notice that $x\in\Omega$ implies that $x_n\ge -L\abs{x'}$ for some $L>0$, so $\abs{x+he_n}\ge h/(1+L^2)^\frac{1}{2}$. Now we apply these estimates and Theorem~\ref{thm:Boundary_Approximation} for some allowable $\alpha>0$ to get
\begin{align*}
\abs{A_2}&\lesssim \sum_{h\le\lambda\le 1}\lambda^{-n}\frac{1}{\lambda^n}\int_{B_\lambda\cap \Omega}\abs{\gamma(x)-\gamma(0)}\,dx + \norm{\gamma}_\infty \\
&\lesssim \sum_{h\le\lambda\le 1}\lambda^{-n+\alpha} + \norm{\gamma}_\infty \\
&= o(h^{-n}).
\end{align*}
The sums here and elsewhere run over dyadic numbers $\lambda=2^k$, for $k$ integer.  This concludes the estimates for the main part.

We turn now to the error term $\div{\gamma\nabla r_h}$. We control it as
\begin{equation*}
\abs{\int \gamma\nabla r_h\cdot\nabla u_h\,dx}\le \norm{r_h}_{H^1}\norm{\div{\gamma\nabla u_h}}_{H^{-1}}.
\end{equation*}
From the \textit{a priori} estimate $\norm{r_h}_{H^1}\le C\norm{\div{\gamma\nabla u_h}}_{H^{-1}}$ we get
\begin{equation*}
\abs{\int \gamma\nabla r_h\cdot\nabla u_h\,dx}\le C\norm{\div{\gamma\nabla u_h}}_{H^{-1}}^2.
\end{equation*}
Since $u_h$ is harmonic, we can bound the operator norm by duality as
\begin{align*}
\abs{\int \gamma\nabla u_h\cdot\nabla\phi} &= \abs{\int (\gamma-\gamma(0))\nabla u_h\cdot\nabla\phi} \\
&\le C\norm{\phi}_{H^1}\Big(\sum_{h\le\lambda\le 1}\lambda^{-n}\frac{1}{\lambda^n}\int_{B_\lambda\cap\Omega}\abs{\gamma-\gamma(0)}^2\,dx + \norm{\gamma}_\infty^2 \Big)^\frac{1}{2} \\
&= \norm{\phi}_{H^1}O(h^{-\frac{n}{2}+\alpha}),
\end{align*}
where $\alpha>0$. Hence, we get
\begin{equation*}
h^n\norm{\div{\gamma\nabla u_h}}_{H^{-1}}^2 = o(1),
\end{equation*}
which concludes the proof of Theorem~\ref{thm:recovery}(A).

We end this section with an estimate for $c_{0,h}=\int_{h^{-1}\Omega}\abs{\nabla u_1}^2$ in \eqref{eq:thm_recovery_A} when the boundary is $C^1$. We fix $\delta\ll 1$ and write
\begin{equation*}
\int_{h^{-1}\Omega}\abs{\nabla u_1}^2 = \int_{h^{-1}(B_\delta\cap\Omega)}\abs{\nabla u_1}^2 + \int_{h^{-1}(B_\delta^c\cap\Omega)}\abs{\nabla u_1}^2.
\end{equation*}
Since $\abs{\nabla u_1}^2$ is integrable, the second term goes to zero as $h\to 0$. We split the first term as
\begin{align*}
\int_{h^{-1}(B_\delta\cap\Omega)}\abs{\nabla u_1}^2 &= \int_{h^{-1}(B_\delta\cap H_+)}\abs{\nabla u_1}^2 + \\
&\hspace*{1cm}+ \Big[\int_{h^{-1}(\Omega\backslash (B_\delta\cap H_+))}\abs{\nabla u_1}^2 - \int_{h^{-1}((B_\delta\cap H_+)\backslash\Omega)}\abs{\nabla u_1}^2\Big] \\
&= B_1 + B_2.
\end{align*}  
The term $B_1$ tends to $\int_{H_+}\abs{\nabla u_1}^2$, and for the second term we have that
\begin{equation*}
\abs{B_2}\lesssim h^{-1}\int_{\abs{x'}\le \delta h^{-1}}\frac{1}{\japan{x'}^{2n}}\abs{\psi(hx')}\,dx' = \frac{1}{h^n}\int_{\abs{x'}\le\delta}\frac{1}{\japan{x'/h}^{2n}}\abs{\psi(x')}\,dx'\to 0.
\end{equation*}
Then $c_{0,h} = \int_{H_+}\abs{\nabla u_1}^2+o(1)$.

\subsection{Normal Derivative at the Boundary}

We will recover $\partial_\nu\gamma$ with the aid of the functions $v_h=\gamma^{-\frac{1}{2}}u_h$. Since $\gamma\in W^{s,p}\cap L^\infty$, then by Gagliardo-Nirenberg we can assume that $\gamma\in W^{s,p}(\Omega)$ for $s>1+\frac{1}{p}$ and $2\le p<\infty$; in fact, if $1<p<2$ and $s>\frac{3}{p}$, then $\gamma\in W^{s',2}(\Omega)$ for $s'=\frac{p}{2}s> 1+\frac{1}{2}$.

We define again a correction function $r_h\in H^1_0(\Omega)$ such that
\begin{equation*}
\div{\gamma\nabla(v_h+r_h)}=0.
\end{equation*}
In this case the functions $f_{1,h}$ in Theorem~\ref{thm:recovery}(B) are $f_{1,h}:=h^\frac{n}{2}v_h|_{\partial\Omega}$; the use of these functions is licit because we already know the value of $\gamma$ at the boundary. We repeat here the arguments in the previous section, but now the computations are longer. 

For the main term we have that
\begin{align}
\int_\Omega \gamma\nabla v_h\cdot \nabla v_h &= \int \abs{\nabla u_h}^2 + 2\int \gamma^\frac{1}{2}\nabla\gamma^{-\frac{1}{2}}u_h\nabla u_h + \int (\gamma^\frac{1}{2}\nabla\gamma^{-\frac{1}{2}})\cdot(\gamma^\frac{1}{2}\nabla\gamma^{-\frac{1}{2}})u_h^2  \notag \\
&=\int \abs{\nabla u_h}^2 - \int\nabla\log\gamma \cdot u_h\nabla u_h + \frac{1}{4}\int \abs{\nabla\log\gamma}^2u_h^2 \notag \\
&= A_1 + A_2 + A_3 \label{eq:Main_Term_Normal_Values}
\end{align}
The principal term in the asymptotic expansion is $h^n A_1=c_{0,h}$; since this term does not involve the conductivity, we can subtract it harmlessly. 

The next term is $A_2$, and we estimate it as
\begin{align*}
\int\nabla\log\gamma \cdot u_h\nabla u_h &= \int_\Omega\nabla\log\gamma(0) \cdot u_h\nabla u_h + \int(\nabla\log\gamma-\nabla\log\gamma(0)) \cdot u_h\nabla u_h  \\
&= \frac{1}{2}\int_{\partial\Omega}u_h^2\nabla\log\gamma(0)\cdot\nu + \int(\nabla\log\gamma-\nabla\log\gamma(0)) \cdot u_h\nabla u_h \\
&=\inlabel{eq:A_2Main} + \inlabel{eq:A_2Error}
\end{align*} 
The term \eqref{eq:A_2Main} contains the information about the normal derivative at the boundary, and it has order $h^{-n+1}$ in the asymptotic expansion. We thus have that
\begin{equation*}
h^n\cdot\eqref{eq:A_2Main} = \frac{h^n}{2}\partial_n\log\gamma(0)\int_{\partial\Omega}u_h^2 + \frac{h^n}{2}\int_{\partial\Omega}u_h^2\nabla\log\gamma(0)\cdot(\nu-e_n).
\end{equation*}
We set $c_{1,h}:=-\frac{1}{2}h^{n-1}\int_{\partial\Omega}u_h^2$ and bound the remaining term as
\begin{equation*}
\abs{\int_{\partial\Omega}u_h^2\nabla\log\gamma(0)\cdot(\nu-e_n)} \lesssim \sum_{h\le\lambda\le 1}\lambda^{-n+1}\frac{1}{\lambda^{n-1}}\int_{B_\lambda\cap\partial\Omega}\abs{\nu-e_n} + 1
\end{equation*}
Since $\nu(0)=-e_n$ is a Lebesgue point, then $G(\lambda):=\frac{1}{\lambda^{n-1}}\int_{B_\lambda\cap\partial\Omega}\abs{\nu-e_n}\xrightarrow{\lambda\to 0} 0$; furthermore, $G$ is uniformly bounded, so by the dominated convergence theorem we get
\begin{equation*}
h^{n-1}\sum_{h\le\lambda\le 1}\lambda^{-n+1}\frac{1}{\lambda^{n-1}}\int_{B_\lambda\cap\partial\Omega}\abs{\nu-e_n} = \sum_{1\le \mu\le h^{-1}}\mu^{-n+1}G(h\mu)\xrightarrow{h\to 0} 0
\end{equation*}
Then we conclude that 
\begin{equation*}
h^n\cdot\eqref{eq:A_2Main} = -c_{1,h}\partial_n\log\gamma(0)\,h + o(h).
\end{equation*}
To control \eqref{eq:A_2Error} we apply Theorem~\ref{thm:Boundary_Approximation} to $\nabla\log\gamma\in W^{s-1,p}(\Omega)$ to get
\begin{align*}
\abs{\eqref{eq:A_2Error}} &\lesssim \sum_{h\le \lambda\le 1}\lambda^{-n+1}\frac{1}{\lambda^n}\int_{B_\lambda\cap\Omega}\abs{\nabla\log\gamma-\nabla\log\gamma(0)}+\int_{B_1^c\cap\Omega}\abs{\nabla\log\gamma-\nabla\log\gamma(0)} \\
&\lesssim \sum_{h\le \lambda\le 1}\lambda^{-n+1+\alpha} + \norm{\nabla\log\gamma}_{2}+\abs{\nabla\log\gamma(0)} \\
&= O(h^{-n+1+\alpha});
\end{align*}
we have thus $h^n\abs{\eqref{eq:A_2Error}}=o(h)$, which allows us to conclude that the term $A_2$ in \eqref{eq:Main_Term_Normal_Values} is
\begin{equation}\label{eq:A_2}
h^nA_2 = c_{1,h}\partial_n\log\gamma(0)\,h + o(h).
\end{equation}

We are left with the error term $A_3$ in \eqref{eq:Main_Term_Normal_Values}. We bound it using the same arguments as above
\begin{align}
A_3 &= \frac{1}{4}\abs{\nabla\log\gamma(0)}^2\int u_h^2 + \frac{1}{4}\int (\abs{\nabla\log\gamma}^2-\abs{\nabla\log\gamma(0)}^2)u_h^2 \notag \\
&\lesssim \int_{\Omega} u_h^2 + \sum_{h\le\lambda\le 1}\lambda^{-n+2}\frac{1}{\lambda^n}\int_{B_\lambda\cap\Omega}(\abs{\nabla\log\gamma}^2-\abs{\nabla\log\gamma(0)}^2) + 1 \notag \\
&= o(h^{-n+1}). \label{eq:A_3}
\end{align}
The estimate we used here to approximate the value of $\abs{\nabla\log\gamma}^2$ at the boundary is not contained in Theorem~\ref{thm:Boundary_Approximation}, and the reader is referred instead to Corollary~\ref{cor:difference_control_square} in the next section. We collect the estimates \eqref{eq:A_2} and \eqref{eq:A_3} to find 
\begin{equation}\label{eq:Derivative_Main_Part}
h^n\int_\Omega \gamma\nabla v_h\cdot \nabla v_h - c_{0,h} = c_{1,h}\partial_n\log\gamma(0)h+ o(h),
\end{equation}
which is what we wanted.

We deal with the error term as before. We have that
\begin{equation*}
\abs{\int \gamma \nabla r_h\cdot\nabla v_h} \le C\norm{\div{\gamma\nabla v_h}}_{H^{-1}}^2.
\end{equation*}
We estimate the norm by duality as
\begin{align*}
\int \gamma\nabla v_h\cdot\nabla \phi &= -\int (\nabla\gamma^\frac{1}{2})u_h\nabla\phi + \int\gamma^\frac{1}{2}\nabla u_h\cdot\nabla\phi \\
&= -\int \nabla\gamma^\frac{1}{2}\cdot\nabla(u_h\phi) \\
&= -\int (\nabla\gamma^\frac{1}{2}-\nabla\gamma^\frac{1}{2}(0))\cdot u_h\nabla\phi + \int (\nabla\gamma^\frac{1}{2}-\nabla\gamma^\frac{1}{2}(0))\cdot \phi\nabla u_h \\
&= E_1 + E_2;
\end{align*}
to get the second and third identities we used the divergence theorem, and the identity $\Delta u_h = 0$. We bound the error term $E_1$ as
\begin{align*}
\abs{E_1} &\lesssim \Big(\sum_{h\le\lambda\le 1} \lambda^{-n+2}\frac{1}{\lambda^n}\int_{B_\lambda\cap\Omega}\abs{\nabla\gamma^\frac{1}{2}-\nabla\gamma^\frac{1}{2}(0)}^2\,dx + 1\Big)^\frac{1}{2}\norm{\phi}_{H^1} \\
&=\norm{\phi}_{H^1}O(h^{-\frac{n}{2}+1})
\end{align*}
To bound $E_2$ we need Hardy's inequality.

\begin{theorem}[Hardy's Inequality]
If $f\in H^1(\R^n)$ for $n\ge 3$, then
\begin{equation}\label{eq:Hardy_High}
\int_{\R^n}\frac{\abs{f}^2}{\abs{x}^2}\le C\norm{f}_{H^1}^2.
\end{equation}
If $f\in H^1_0(A_{\delta, R})$ for $A_{\delta,R}:=\{\delta <\abs{x}< R\}\subset\R^2$, then
\begin{equation}\label{eq:Hardy_plane}
\int_{\R^2}\frac{\abs{f}^2}{\abs{x}^2}\le C\Big(\log\big(\frac{R}{\delta}\big)\Big)^2\norm{f}_{H^1}^2.
\end{equation}
\end{theorem}

A beautiful proof can be found in \cite[sec. 2]{MR1616905}. The inequality \eqref{eq:Hardy_plane} is not there, but it follows after minor changes.

For $n\ge 3$, we apply Hardy's inequality with the weight $\abs{x+he_n}^{-2}$ to get
\begin{align*}
\abs{E_2}&\le (\int_\Omega \abs{\nabla\gamma^\frac{1}{2}-\nabla\gamma^\frac{1}{2}(0)}^2\abs{x+he_n}^2\abs{\nabla u_h}^2)^\frac{1}{2}\norm{\phi}_{H^1} \\
&\le \Big(\sum_{h\le\lambda\le 1}\lambda^{-n+2}\frac{1}{\lambda^n}\int_{B_\lambda\cap\Omega}\abs{\nabla\gamma^\frac{1}{2}-\nabla\gamma^\frac{1}{2}(0)}^2 + 1 \Big)^\frac{1}{2}\norm{\phi}_{H^1} \\
&= \norm{\phi}_{H^1}O(h^{-\frac{n}{2}+1}).
\end{align*}
For $n=2$ we get $\abs{E_2}=\norm{\phi}_{H^1}O(\log h^{-1})$. Hence, we conclude that
\begin{equation*}
h^n\int \gamma \nabla r_h\cdot\nabla v_h = o(h).
\end{equation*}
With this and \eqref{eq:Derivative_Main_Part} we get Theorem~\ref{thm:recovery}(B).

\section{Lebesgue Points at the Boundary} \label{sec:Approximation}

In this section we prove Theorem~\ref{thm:Boundary_Approximation}. The main tool to control the value of a function at the boundary is the next theorem. 

\begin{theorem}\label{thm:difference_control}
If $f\in B^{s,p}(\R^n)$ for $1+\frac{1}{p}>s>\frac{1}{p}$, and $\Gamma$ is the graph of a Lipschitz function, then for $0\le\alpha<s-\frac{1}{p}$ it holds that
\begin{equation}\label{eq:difference_control}
\Big(\int_\Gamma\int_{\abs{x-y}\le 1}\frac{\abs{f(x)-f(y)}^q}{\abs{x-y}^{n+\alpha q}}\,dx d\Gamma(y)\Big)^\frac{1}{q} \le C_\alpha\norm{f}_{s,p},\quad \text{where } 1\le q\le p<\infty.
\end{equation}
\end{theorem}
As a consequence of this theorem we get Theorem~\ref{thm:Boundary_Approximation}, which we restate and prove here.

\begin{customthm}{\ref{thm:Boundary_Approximation}}
{\it
Suppose that $\Omega\subset\R^n$ is a domain with Lipschitz boundary. If $f\in B^{s,p}(\Omega)$ for $1+\frac{1}{p}>s>\frac{1}{p}$, then for $0\le\alpha<s-\frac{1}{p}$ and for $y\in\partial\Omega$ a.e. it holds that
\begin{equation}
\Big(\frac{1}{r^n}\int_{B_r(y)\cap\Omega}\abs{f(x)-f(y)}^q\,dx\Big)^\frac{1}{q} \le Cr^\alpha,\quad \text{where } r\le 1 \text{ and } 1\le q\le p.
\end{equation}
The constant $C$ depends on $y$.
}
\end{customthm}
\begin{proof}
By definition there is some $g\in B^{s,p}(\R^n)$ that extends $f$, and 
\begin{equation}
\int_{B_r(y)\cap\Omega}\abs{f(x)-f(y)}^q\,dx\le \int_{B_r(y)}\abs{g(x)-g(y)}^q\,dx.
\end{equation}
We divide the boundary into pieces $\Gamma\subset\partial\Omega$, where $\Gamma$ is the graph of a Lipschitz function. Since
\begin{equation*}
\frac{1}{r^{n+\alpha q}}\int_{B_r(y)}\abs{g(x)-g(y)}^q\,dx \le \int_{\abs{x-y}\le 1}\frac{\abs{g(x)-g(y)}^q}{\abs{x-y}^{n+\alpha q}}\,dx,
\end{equation*} 
and since the term at the right is finite for $y\in\Gamma\subset\partial\Omega$ a.e. by Theorem~\ref{thm:difference_control}, then we have that
\begin{equation*}
\frac{1}{r^n}\int_{B_r(y)}\abs{g(x)-g(y)}^q\,dx\le Cr^{\alpha q},
\end{equation*}
and the statement of the theorem follows.
\end{proof}

In Theorem~\ref{thm:difference_control} we assumed implicitly that $f\in B^{s,p}(\R^n)$, for $s>1/p$, is well defined in $\Gamma$, but this set has measure zero, so this need some justification. Let $f=\sum_{\lambda\ge 1}P_\lambda f$ be a Littlewood-Paley decomposition, where $(P_\lambda f)^\wedge:=m_\lambda \widehat{f}$, and $m_\lambda(\xi):=m(\xi/\lambda)$ for some smooth multiplier $m$ supported in frequencies $\abs{\xi}\sim 1$; for low frequencies we take a function $m_1$ supported in $\abs{\xi}\lesssim 1$. We choose the representative of $f$ given by $\lim_{M\to\infty}P_{\le M}f(x):= \lim_{M\to\infty}\sum_{1\le \lambda\le M} P_\lambda f(x)$. The following theorem justifies this choice.

\begin{lemma}
Suppose that $\Gamma$ is the graph of a Lipschitz function. If $f\in B^{s,p}(\R^n)$ for $s>1/p$, then $\lim_{M\to\infty}P_{\le M}f(y)$ exits for $y\in\Gamma$ a.e.
\end{lemma}
\begin{proof}
The set of divergence is
\begin{equation*}
\{y\in\Gamma\mid \limsup_{N\le M\to\infty}\abs{P_{N\le M}f(y)}>0\}=\bigcup_{\mu>0}\{y\in\Gamma\mid \limsup_{N\le M\to\infty}\abs{P_{N\le M}f(y)}>\mu\},
\end{equation*}
then it suffices to prove that each set at the right has measure zero. For each one of these sets and for every $A>0$ we have that
\begin{equation*}
\{y\in\Gamma\mid \limsup_{N\le M\to\infty}\abs{P_{N\le M}f(y)}>\mu\}\subset \{y\in\Gamma\mid \sup_{A\le N\le M}\abs{P_{N\le M}f(y)}>\mu\},
\end{equation*} 
so we only need to show that the sets at the right are as small as we please if we choose $A\gg 1$. We bound their measure as
\begin{align*}
\abs{\{\sup_{A\le N\le M}\abs{P_{N\le M}f(y)}>\mu\}}&\le\abs{\{\sum_{\lambda\ge A}\abs{P_\lambda f(y)}>\mu\}} \\
&\le \frac{1}{\mu^p}\norm{\sum_{\lambda\ge A}\abs{P_\lambda f}}_{L^p(\Gamma)}^p.
\end{align*}
We use the triangle inequality, the trace inequality $\norm{P_\lambda f}_{L^p(\Gamma)}\le C\lambda^\frac{1}{p}\norm{f}_{L^p(\R^n)}$, which we will prove in Lemma~\ref{lemm:trace_bound} below, and Hölder to bound the last term as
\begin{equation*}
\abs{\{\sup_{A\le N\le M}\abs{P_{N\le M}f(y)}>\mu\}}\le C\frac{A^{1-sp}}{\mu^p}\norm{f}_{s,p}^p.
\end{equation*}
Since $1-sp<0$, then the right hand side goes to zero as $A\to \infty$.
\end{proof}

\begin{lemma}[The Trace Inequality]\label{lemm:trace_bound}
Suppose that $m_\lambda(\xi):=m(\xi/\lambda)$ is a smooth multiplier supported in frequencies $\abs{\xi}\sim \lambda$, and that $(P_\lambda f)^\wedge=m_\lambda\hat{f}$ is the associated projection. If $\Gamma$ is the graph of a Lipschitz function, then
\begin{equation}\label{eq:lemm_trace_bound}
\norm{P_\lambda f}_{L^q(\Gamma)}\le C\lambda^\frac{1}{p}\norm{f}_{L^p(\R^n)},\quad \text{for } 1\le q\le p.
\end{equation}  
\end{lemma}
\begin{proof}
We interpolate between $(q,p)=(\infty,\infty)$ and $=(1,r)$ for $r=p/q$. For the first point we have that
\begin{equation*}
\norm{P_\lambda f}_{L^\infty(\Gamma)}\le\norm{\widecheck{m}_\lambda}_1\norm{f}_\infty = C\norm{f}_{L^\infty(\R^n)},
\end{equation*}
where $C$ does not depend on $\lambda$.

For the point $(q,p)=(1,r)$ we have that
\begin{align}\label{eq:lemm_trace_measure}
\norm{P_\lambda f}_{L^1(\Gamma)} &\le \int\abs{f(z)}\int_\Gamma \abs{\widecheck{m}_\lambda(y-z)}\,d\Gamma(y)dz \notag \\
&\le \norm{f}_{L^r(\R^n)}\norm{\int_\Gamma \abs{\widecheck{m}_\lambda(y-z)}\,d\Gamma(y)}_{L_z^{r'}}.
\end{align}
By the smoothness of $m$ we have that $\abs{\widecheck{m}_\lambda}\le C\lambda^n\sum_{\mu\ge \lambda^{-1}} (\mu\lambda)^{-N}\ind_{B_\mu}$, where $N\gg 1$. Then
\begin{equation*}
\norm{\int_\Gamma \abs{\widecheck{m}_\lambda(y-z)}\,d\Gamma(y)}_{r'}\le \lambda^n\sum_{\mu\ge \lambda^{-1}}(\mu\lambda)^{-N}\norm{\int \ind_{B_\mu}(y-z)d\Gamma(y)}_{r'},
\end{equation*}
and we define the functions $G_\mu(z):=\int \ind_{B_\mu}(y-z)d\Gamma(y)=\abs{B_\mu(z)\cap\Gamma}$; we use here the induced measure in $\Gamma$. If $N_\mu(\Gamma)$ denotes the $\mu$-neighborhood of $\Gamma$, then we have the following estimates
\begin{equation*}
G_\mu(z)\lesssim 
\begin{cases}
\mu^{n-1}\ind_{N_\mu(\Gamma)} &\text{if } \mu\le \text{diam}(\Gamma) \\
\ind_{N_\mu(\Gamma)} &\text{otherwise}.
\end{cases}
\end{equation*}
Hence,
\begin{align*}
\norm{\int_\Gamma \abs{\widecheck{m}_\lambda(y-z)}\,d\Gamma(y)}_{r'}&\le \lambda^n\sum_{\mu\ge \lambda^{-1}}(\mu\lambda)^{-N}\norm{G_\mu}_{r'} \\
&\le C\lambda^n\Big(\sum_{\lambda^{-1}\le\mu \le \text{diam}(\Gamma)}(\mu\lambda)^{-N}\mu^{n-\frac{1}{r}} + \sum_{\text{diam}(\Gamma)\le\mu}(\mu\lambda)^{-N}\mu^\frac{n}{r'}\Big)\\
&\le C\lambda^\frac{1}{r}.
\end{align*}
We replace this bound in \eqref{eq:lemm_trace_measure} to get $\norm{P_\lambda f}_{L^1(\Gamma)}\le C\lambda^\frac{1}{r}\norm{f}_r$, which concludes the proof.
\end{proof}

Now we are ready to prove Theorem~\ref{thm:difference_control}.

\begin{proof}[Proof of Theorem~\ref{thm:difference_control}]
We follow the arguments in \cite[chp. 5]{MR0290095}. After a change of variables we can write \eqref{eq:difference_control} as
\begin{equation}\label{eq:modulus_norm}
\Big(\int_{\abs{x}\le 1}\frac{\norm{f(x+y)-f(y)}_{L_y^q(\Gamma)}^q}{\abs{x}^{n+\alpha q}}\,dx\Big)^\frac{1}{q} \le C_\alpha\norm{f}_{s,p}.
\end{equation}
To estimate $\norm{f(x+y)-f(y)}_{L_y^q(\Gamma)}$ we write the difference as
\begin{multline*}
f(x+y)-f(y)=(f(x+y) - P_{\le M}f(x+y))+ \\ 
+(P_{\le M}f(x+y)-P_{\le M}f(y))+(P_{\le M}f(y)-f(y)).
\end{multline*}
For the first term $f(x+y) - P_{\le M}f(x+y)=P_{>M}f(x+y)$, we start by applying the triangle inequality to get
\begin{equation*}
\norm{P_{>M}f(x+y)}_{L_y^q(\Gamma)}\le \sum_{\lambda>M}\norm{P_\lambda f(x+y)}_{L_y^q(\Gamma)};
\end{equation*}
now we use the trace inequality $\norm{P_\lambda g}_{L^q(\Gamma)}\le C\lambda^\frac{1}{p}\norm{g}_{L^p(\R^n)}$ in Lemma~\ref{lemm:trace_bound} to get
\begin{equation}\label{eq:high_frequencies}
\norm{P_{>M}f(x+y)}_{L_y^q(\Gamma)}\le C\sum_{\lambda>M}\lambda^\frac{1}{p}\norm{P_\lambda f}_{L^p(\R^n)}\le CM^{\frac{1}{p}-s}\norm{f}_{s,p}.
\end{equation}
We estimate the difference $P_{\le M}f(y)-f(y)$ in the same way.

For the difference $P_{\le M}f(x+y)-P_{\le M}f(y)$ we use the smoothness of the projection to write
\begin{equation*}
P_{\le M}f(x+y)-P_{\le M}f(y)=x\cdot\int_0^1\nabla P_{\le M}f(tx+y)\,dt.
\end{equation*}
The multiplier of $\partial_iP_\lambda$ is $\xi_im(\xi/\lambda)=\lambda\tilde{m}(\xi/\lambda)$, where $\tilde{m}(\xi):=\xi_im(\xi)$ is a smooth function supported in $\abs{\xi}\sim 1$. By Minkowski and the trace inequality we have that
\begin{align}\label{eq:low_frequencies}
\norm{P_{\le M}f(x+y)-P_{\le M}f(y)}_{L_y^q(\Gamma)}&\le\abs{x}\sum_{\lambda\le M}\lambda\int_0^1\norm{\tilde{P}_{\lambda}f(tx+y)}_{L_y^q(\Gamma)}\,dt \notag \\
&\le \abs{x}\sum_{\lambda\le M}\lambda^{1+\frac{1}{p}}\norm{P_{\lambda}f}_{L^p(\R^n)} \notag \\
&\le C\abs{x}M^{1+\frac{1}{p}-s}\norm{f}_{s,p}.
\end{align}
We bound $\norm{f(x+y)-f(y)}_{L_y^q(\Gamma)}$ using \eqref{eq:high_frequencies} and \eqref{eq:low_frequencies} to get
\begin{align*}
\norm{f(x+y)-f(y)}_{L_y^q(\Gamma)}&\le C(M^{\frac{1}{p}-s}+\abs{x}M^{1+\frac{1}{p}-s})\norm{f}_{s,p} \\
&\le C\abs{x}^{s-\frac{1}{p}}\norm{f}_{s,p};
\end{align*}
we obtained the last inequality by choosing $M\sim\abs{x}^{-1}$. We insert this bound into the term at the left of \eqref{eq:modulus_norm} to get 
\begin{equation*}
\Big(\int_{\abs{x}\le 1}\frac{\norm{f(x+y)-f(y)}_{L_y^q(\Gamma)}^q}{\abs{x}^{n+\alpha q}}\,dx\Big)^\frac{1}{q} \le C\Big(\int_{\abs{x}\le 1}\abs{x}^{-n+q(s-\frac{1}{p}-\alpha)}\,dx\Big)^\frac{1}{q}\norm{f}_{s,p},
\end{equation*}
and the last integral is bounded whenever $\alpha<s-\frac{1}{p}$.

\end{proof}

In Section~\ref{sec:Recovery} we needed also the following result.
\begin{corollary}\label{cor:difference_control_square}
Suppose that $2\le p<\infty$. If $f\in B^{s,p}(\R^n)$ for $1+\frac{1}{p}>s>\frac{1}{p}$, and $\Gamma$ is the graph of a Lipschitz function, then for $0\le\alpha<s-\frac{1}{p}$ it holds that
\begin{equation}\label{eq:difference_control_square}
\int_\Gamma\int_{\abs{x-y}\le 1}\frac{\abs{f^2(x)-f^2(y)}}{\abs{x-y}^{n+\alpha}}\,dx d\Gamma(y) \le C_\alpha\norm{f}_{s,p}^2.
\end{equation}
Consequently
\begin{equation}\label{eq:cor_Square_approach}
\frac{1}{r^n}\int_{B_r(y)\cap\Omega}\abs{f^2(x)-f^2(y)}\,dx\le C_y r^\alpha \quad\text{for } y\in\partial\Omega \text{ a.e.}
\end{equation}
\end{corollary}
\begin{proof}
We write again the expected estimate as
\begin{equation*}
\int_{\abs{x}\le 1}\frac{\norm{f^2(x+y)-f^2(y)}_{L_y^1(\Gamma)}}{\abs{x}^{n+\alpha }}\,dx\le C_\alpha\norm{f}_{s,p}^2.
\end{equation*}
By Hölder and by the trace theorem we can bound the difference as
\begin{equation*}
\norm{f^2(x+y)-f^2(y)}_{L_y^1(\Gamma)}\le 2\norm{f}_{s,p}\norm{f(x+y)-f(y)}_{L_y^{p'}(\Gamma)}.
\end{equation*}
Then by Hölder and by Theorem~\ref{thm:difference_control} we get
\begin{align*}
\int_{\abs{x}\le 1}\frac{\norm{f^2(x+y)-f^2(y)}_{L_y^1(\Gamma)}}{\abs{x}^{n+\alpha }}\,dx &\le C\norm{f}_{s,p}\int_{\abs{x}\le 1}\frac{\norm{f(x+y)-f(y)}_{L_y^{p'}(\Gamma)}}{\abs{x}^{n+\alpha}}\,dx \\
&\le C\norm{f}_{s,p}\Big(\int_{\abs{x}\le 1}\frac{\norm{f(x+y)-f(y)}_{L_y^{p'}(\Gamma)}^{p'}}{\abs{x}^{n+(\alpha+\varepsilon)p'}}\,dx\Big)^\frac{1}{p'} \\
&\le C\norm{f}_{s,p}^2,
\end{align*}
where $\varepsilon\ll 1$. This is inequality \eqref{eq:difference_control_square}, and \eqref{eq:cor_Square_approach} follows.
\end{proof}

\subsection*{Notations}

\begin{itemize}
\item Various: $B_r$ is a ball or radius $r$, usually centered at zero, and we make the center explicit by writing $B_r(y)$. $H_+$ is the upper half plane. When we write $\sum_{\lambda}$, we are summing over dyadic numbers $\lambda=2^k$ for $k$ integer.
\item If $E$ is a set in some measure space, then $\abs{E}$ denotes the size of $E$ for the corresponding measure.
\item We write $A=O(B)$, or $A\lesssim B$, if $A\le CB$ for some $C>0$; if $cB\le A\le CB$ for some constants $0<c,C$, then $A\sim B$. We write $A\ll 1$ if $A$ is sufficiently small. We write $A=o(B)$ if $A$ and $B$ are functions of $h$ and $\lim_{h\to 0}A/B=0$. 
\item Projections: for a dyadic number $\lambda$ we define the projection $(P_\lambda f)^\wedge = m_\lambda\widehat{f}$, where $m_\lambda(\xi)=m(\xi/\lambda)$ and $m$ is a smooth multiplier supported in frequencies $\abs{\xi}\sim 1$; for $\lambda=1$ we take instead $m_1$ supported in $\abs{\xi}\lesssim 1$.
\item Function spaces: Let $1< p<\infty$. We denote by $B^{s,p}(\R^n)$ the Besov space of distributions $f$ for which
\begin{equation*}
\norm{f}_{B^{s,p}}^p := \norm{P_1f}_p^p + \sum_{\lambda\ge 1}\lambda^{sp}\norm{P_\lambda f}^p_p <\infty.
\end{equation*}
The Sobolev-Slobodeskij space $W^{s,p}(\R^n)$ equals $B^{s,p}(\R^n)$ for $s\neq$ integer. For $s$ integer, $W^{s,p}(\R^n)$ is the space of distributions $f$ for which
\begin{equation*}
\norm{f}_{W^{s,p}} := \sum_{\abs{\alpha}\le s}\norm{D^\alpha f}_p <\infty.
\end{equation*} 
We set $H^s(\R^n):=W^{s,2}(\R^n)$. The spaces $B^{s,p}(\Omega)$, $W^{s,p}(\Omega)$ and $H^s(\Omega)$ are defined by restriction of functions in $\R^n$ to $\Omega$. The space $H^s_0(\Omega)$ is the completion in the norm $H^s(\R^n)$ of test functions compactly supported in $\Omega$. 
\end{itemize}

\bibliographystyle{plain}
\bibliography{../Calderon_Problem,../Restriction,../Function_Space}

\begin{thebibliography}{10}

\bibitem{MR1047569}
G.~Alessandrini.
\newblock Singular solutions of elliptic equations and the determination of
  conductivity by boundary measurements.
\newblock {\em J. Differential Equations}, 84(2):252--272, 1990.

\bibitem{MR3813967}
H.~Brezis and P.~Mironescu.
\newblock Gagliardo-{N}irenberg inequalities and non-inequalities: the full
  story.
\newblock {\em Ann. Inst. H. Poincar\'{e} Anal. Non Lin\'{e}aire},
  35(5):1355--1376, 2018.

\bibitem{MR1881563}
R.~Brown.
\newblock Recovering the conductivity at the boundary from the {D}irichlet to
  {N}eumann map: a pointwise result.
\newblock {\em J. Inverse Ill-Posed Probl.}, 9(6):567--574, 2001.

\bibitem{MR2026763}
R.~Brown and R.~Torres.
\newblock Uniqueness in the inverse conductivity problem for conductivities
  with {$3/2$} derivatives in {$L^p,\ p>2n$}.
\newblock {\em J. Fourier Anal. Appl.}, 9(6):563--574, 2003.

\bibitem{MR590275}
A.~Calder\'{o}n.
\newblock On an inverse boundary value problem.
\newblock In {\em Seminar on {N}umerical {A}nalysis and its {A}pplications to
  {C}ontinuum {P}hysics ({R}io de {J}aneiro, 1980)}, pages 65--73. Soc. Brasil.
  Mat., Rio de Janeiro, 1980.

\bibitem{MR1616905}
J.~P. Garc\'{\i}a~Azorero and I.~Peral~Alonso.
\newblock Hardy inequalities and some critical elliptic and parabolic problems.
\newblock {\em J. Differential Equations}, 144(2):441--476, 1998.

\bibitem{MR3397029}
B.~Haberman.
\newblock Uniqueness in {C}alder\'{o}n's problem for conductivities with
  unbounded gradient.
\newblock {\em Comm. Math. Phys.}, 340(2):639--659, 2015.

\bibitem{HKL}
S.~Ham, Y.~Kwon, and S.~Lee.
\newblock Uniqueness in the {C}alder\'{o}n problem and bilinear restriction
  estimates.
\newblock arXiv:1903.09382v2 [math.AP], 2019.

\bibitem{MR739921}
R.~Kohn and M.~Vogelius.
\newblock Determining conductivity by boundary measurements.
\newblock {\em Comm. Pure Appl. Math.}, 37(3):289--298, 1984.

\bibitem{MR884984}
J.~Marschall.
\newblock The trace of {S}obolev-{S}lobodeckij spaces on {L}ipschitz domains.
\newblock {\em Manuscripta Math.}, 58(1-2):47--65, 1987.

\bibitem{PV}
F.~Ponce-Vanegas.
\newblock The bilinear strategy for {C}alder\'{o}n's problem.
\newblock arXiv:1908.04050 [math.AP], 2019.

\bibitem{MR0290095}
E.~Stein.
\newblock {\em Singular integrals and differentiability properties of
  functions}.
\newblock Princeton Mathematical Series, No. 30. Princeton University Press,
  Princeton, N.J., 1970.

\bibitem{MR873380}
J.~Sylvester and G.~Uhlmann.
\newblock A global uniqueness theorem for an inverse boundary value problem.
\newblock {\em Ann. of Math. (2)}, 125(1):153--169, 1987.

\bibitem{MR924684}
J.~Sylvester and G.~Uhlmann.
\newblock Inverse boundary value problems at the boundary---continuous
  dependence.
\newblock {\em Comm. Pure Appl. Math.}, 41(2):197--219, 1988.

\bibitem{MR781540}
H.~Triebel.
\newblock {\em Theory of function spaces}, volume~78 of {\em Monographs in
  Mathematics}.
\newblock Birkh\"{a}user Verlag, Basel, 1983.

\end{thebibliography}

\end{document}